\documentclass[11pt,leqno]{article}
\usepackage[margin=1in]{geometry} 
\usepackage{amssymb,amsfonts,amsmath,bbm,mathrsfs,stmaryrd,mathtools,mathabx}
\usepackage{xcolor}
\usepackage{url}

\usepackage{extarrows}

\usepackage[shortlabels]{enumitem}
\usepackage{tensor}

\usepackage{xr}
\usepackage[T1]{fontenc}
\usepackage[utf8]{inputenc}
\usepackage[colorlinks,
linkcolor=black!75!red,
citecolor=blue,
pdftitle={},
pdfproducer={pdfLaTeX},
pdfpagemode=None,
bookmarksopen=true,
bookmarksnumbered=true,
backref=page]{hyperref}

\usepackage{tikz}
\usetikzlibrary{arrows,calc,decorations.pathreplacing,decorations.markings,decorations.shapes,intersections,shapes.geometric,through,fit,shapes.symbols,positioning,decorations.pathmorphing}

\newcommand{\pfr}{\hyperref[th:pfr.intro]{PFR}}
\newcommand{\pfrp}{\hyperref[th:pfrp.intro]{PFR$_{\perp}$}}
\newcommand{\clr}{\hyperref[th:clr]{CLR}}

\usepackage{amsmath} % (mathtools is fine too; it loads/extends amsmath)

\makeatletter
\newlength\zig@L
\newlength\zig@La
\newlength\zig@Lb

\newcommand{\xzigrightarrow}[2][]{%
  \mathrel{%
    \settowidth{\zig@La}{$\scriptstyle #2$}%
    \settowidth{\zig@Lb}{$\scriptstyle #1$}%
    \zig@L=\zig@La\relax
    \ifdim\zig@Lb>\zig@L \zig@L=\zig@Lb\fi
    \advance\zig@L by 2.2em\relax
    \tikz[baseline=-0.65ex]{%
      \draw[->,
            line cap=round,
            decorate,
            decoration={zigzag,segment length=4pt,amplitude=1.1pt}]%
        (0,0) -- (\zig@L,0)
        node[midway,above=2pt] {$\scriptstyle #2$}%
        \if\relax\detokenize{#1}\relax\else
          node[midway,below=2pt] {$\scriptstyle #1$}%
        \fi
      ;
    }%
  }%
}
\makeatother

\makeatletter
\newcommand{\squigjoin}{1mu} % tune this: 0.5mu, 1mu, 1.5mu, ...

\def\sqleft@{\sim}                    % no overlap here
\def\sqmid@{\sim\mkern-\squigjoin}    % overlap only between repeated mids

\def\rightsquigarrowfill@{%
  \arrowfill@{\sqleft@}{\sqmid@}{\mkern-4mu\succ}%
}

\newcommand{\xrightsquigarrow}[2][]{%
  \ext@arrow 0359\rightsquigarrowfill@{#1}{#2}%
}
\makeatother

\makeatletter
\newcommand*\circled[1]{\tikz[baseline=(char.base)]{
    \node[shape=circle, draw, inner sep=0pt, 
    minimum height={\f@size},] (char) {\vphantom{WAH1g}#1};}}
\makeatother

\makeatletter
\DeclareRobustCommand\widecheck[1]{{\mathpalette\@widecheck{#1}}}
\def\@widecheck#1#2{%
    \setbox\z@\hbox{\m@th$#1#2$}%
    \setbox\tw@\hbox{\m@th$#1%
       \widehat{%
          \vrule\@width\z@\@height\ht\z@
          \vrule\@height\z@\@width\wd\z@}$}%
    \dp\tw@-\ht\z@
    \@tempdima\ht\z@ \advance\@tempdima2\ht\tw@ \divide\@tempdima\thr@@
    \setbox\tw@\hbox{%
       \raise\@tempdima\hbox{\scalebox{1}[-1]{\lower\@tempdima\box
\tw@}}}%
    {\ooalign{\box\tw@ \cr \box\z@}}}
\makeatother

\usepackage{braket}

\usepackage[amsmath,thmmarks,hyperref]{ntheorem}
\usepackage{cleveref}

\newcommand\nthalias[1]{\AddToHook{env/#1/begin}{\crefalias{lemma}{#1}}}

\nthalias{definition}
\nthalias{example}
\nthalias{examples}
\nthalias{remark}
\nthalias{remarks}
\nthalias{convention}
\nthalias{notation}
\nthalias{construction}
\nthalias{sketch}
\nthalias{theoremN}
\nthalias{propositionN}
\nthalias{corollaryN}
\nthalias{lemma}
\nthalias{proposition}
\nthalias{corollary}
\nthalias{theorem}
\nthalias{conjecture}
\nthalias{question}
\nthalias{assumption}

\creflabelformat{enumi}{#2#1#3}

\crefname{section}{Section}{Sections}
\crefformat{section}{#2Section~#1#3} 
\Crefformat{section}{#2Section~#1#3} 

\crefname{subsection}{\S}{\S\S}
\AtBeginDocument{%
  \crefformat{subsection}{#2\S#1#3}%
  \Crefformat{subsection}{#2\S#1#3}% 
}

\crefname{subsubsection}{\S}{\S\S}
\AtBeginDocument{%
  \crefformat{subsubsection}{#2\S#1#3}%
  \Crefformat{subsubsection}{#2\S#1#3}% 
}

\theoremstyle{plain}

\newtheorem{lemma}{Lemma}[section]

\newtheorem{theorem}[lemma]{Theorem}

\theoremstyle{plain}
\theoremnumbering{Alph}

\theoremstyle{plain}
\theorembodyfont{\upshape}
\theoremsymbol{\ensuremath{\blacklozenge}}

\newtheorem{definition}[lemma]{Definition}

\newtheorem{remark}[lemma]{Remark}
\newtheorem{remarks}[lemma]{Remarks}

\newtheorem{notation}[lemma]{Notation}
\newtheorem{construction}[lemma]{Construction}

\crefname{definition}{definition}{definitions}
\crefformat{definition}{#2definition~#1#3} 
\Crefformat{definition}{#2Definition~#1#3} 

\crefname{ex}{example}{examples}
\crefformat{example}{#2example~#1#3} 
\Crefformat{example}{#2Example~#1#3} 

\crefname{exs}{example}{examples}
\crefformat{examples}{#2example~#1#3} 
\Crefformat{examples}{#2Example~#1#3} 

\crefname{remark}{remark}{remarks}
\crefformat{remark}{#2remark~#1#3} 
\Crefformat{remark}{#2Remark~#1#3} 

\crefname{remarks}{remark}{remarks}
\crefformat{remarks}{#2remark~#1#3} 
\Crefformat{remarks}{#2Remark~#1#3} 

\crefname{convention}{convention}{conventions}
\crefformat{convention}{#2convention~#1#3} 
\Crefformat{convention}{#2Convention~#1#3} 

\crefname{notation}{notation}{notations}
\crefformat{notation}{#2notation~#1#3} 
\Crefformat{notation}{#2Notation~#1#3} 

\crefname{table}{table}{tables}
\crefformat{table}{#2table~#1#3} 
\Crefformat{table}{#2Table~#1#3}

\crefname{lemma}{lemma}{lemmas}
\crefformat{lemma}{#2lemma~#1#3} 
\Crefformat{lemma}{#2Lemma~#1#3} 

\crefname{proposition}{proposition}{propositions}
\crefformat{proposition}{#2proposition~#1#3} 
\Crefformat{proposition}{#2Proposition~#1#3} 

\crefname{propositionN}{proposition}{propositions}
\crefformat{propositionN}{#2proposition~#1#3} 
\Crefformat{propositionN}{#2Proposition~#1#3} 

\crefname{corollary}{corollary}{corollaries}
\crefformat{corollary}{#2corollary~#1#3} 
\Crefformat{corollary}{#2Corollary~#1#3} 

\crefname{corollaryN}{corollary}{corollaries}
\crefformat{corollaryN}{#2corollary~#1#3} 
\Crefformat{corollaryN}{#2Corollary~#1#3} 

\crefname{theorem}{theorem}{theorems}
\crefformat{theorem}{#2theorem~#1#3} 
\Crefformat{theorem}{#2Theorem~#1#3} 

\crefname{theoremN}{theorem}{theorems}
\crefformat{theoremN}{#2theorem~#1#3} 
\Crefformat{theoremN}{#2Theorem~#1#3} 

\crefname{enumi}{}{}
\crefformat{enumi}{#2#1#3}
\Crefformat{enumi}{#2#1#3}

\crefname{assumption}{assumption}{Assumptions}
\crefformat{assumption}{#2assumption~#1#3} 
\Crefformat{assumption}{#2Assumption~#1#3} 

\crefname{construction}{construction}{Constructions}
\crefformat{construction}{#2construction~#1#3} 
\Crefformat{construction}{#2Construction~#1#3} 

\crefname{sketch}{sketch}{Sketches}
\crefformat{sketch}{#2sketch~#1#3} 
\Crefformat{sketch}{#2Sketch~#1#3} 

\crefname{question}{question}{Questions}
\crefformat{question}{#2question~#1#3} 
\Crefformat{question}{#2Question~#1#3} 

\crefname{equation}{}{}
\crefformat{equation}{(#2#1#3)} 
\Crefformat{equation}{(#2#1#3)}

\numberwithin{equation}{section}

\theoremstyle{nonumberplain}
\theoremsymbol{\ensuremath{\blacksquare}}

\newtheorem{proof}{Proof}
\newcommand\pf[1]{\newtheorem{#1}{Proof of \Cref{#1}}}

\newcommand\bC{{\mathbb C}}

\newcommand\bF{{\mathbb F}}
\newcommand\bG{{\mathbb G}}

\newcommand\bP{{\mathbb P}}

\newcommand\bR{{\mathbb R}}

\newcommand\bZ{{\mathbb Z}}

\newcommand\cC{{\mathcal C}}
\newcommand\cD{{\mathcal D}}

\newcommand\cL{{\mathcal L}}

\newcommand\cN{{\mathcal N}}

\newcommand\cP{{\mathcal P}}

\newcommand\cT{{\mathcal T}}

\newcommand\cY{{\mathcal Y}}

\newcommand\wt{\widetilde}
\newcommand\wh{\widehat}

\DeclareMathOperator{\Ad}{Ad}
\DeclareMathOperator{\id}{id}

\DeclareMathOperator{\spn}{\mathrm{spn}}

\DeclareMathOperator{\GL}{GL}

\newcommand{\cat}[1]{\textsc{#1}}

\newcommand{\qedhere}{\mbox{}\hfill\ensuremath{\blacksquare}}

\title{Frame eversion and contextual geometric rigidity}
\author{Alexandru Chirvasitu}

\begin{document}

\date{}

\newcommand{\Addresses}{{% additional braces for segregating \footnotesize
  \bigskip
  \footnotesize

  \textsc{Department of Mathematics, University at Buffalo}
  \par\nopagebreak
  \textsc{Buffalo, NY 14260-2900, USA}  
  \par\nopagebreak
  \textit{E-mail address}: \texttt{achirvas@buffalo.edu}

  % % \medskip
  % % 
  % % \textsc{Department of Mathematics, INSTITUTION}
  % % \par\nopagebreak
  % % \textsc{ADDRESS}
  % % \par\nopagebreak
  % % \textit{E-mail address}: \texttt{??}
  % % 

}}

\maketitle

\begin{abstract}
  We prove rigidity results describing contextually-constrained maps defined on Grassmannians and manifolds of ordered independent line tuples in finite-dimensional vector or Hilbert spaces. One statement in the spirit of the Fundamental Theorem of Projective Geometry classifies maps between full Grassmannians of two $n$-dimensional Hilbert spaces, $n\ge 3$, preserving dimension and lattice operations for pairs with commuting orthogonal projections, as precisely those induced by semilinear injections unique up to scaling.
  
  In a different but related direction, denote the manifolds of ordered orthogonal (linearly-independent) $n$-tuples of lines in an $n$-dimensional Hilbert space $V$ by $\mathbb{F}^{\perp}(V)$ (respectively $\mathbb{F}(V)$) and, for partitions $\pi$ of the set $\{1..n\}$, call two tuples $\pi$-linked if the spans along $\pi$-blocks agree. A Wigner-style rigidity theorem proves that the symmetric maps $\mathbb{F}^{\perp}(\mathbb{C}^n)\to \mathbb{F}(\mathbb{C}^n)$, $n\ge 3$ respecting $\pi$-linkage are precisely those induced by semilinear injections, hence by linear or conjugate-linear maps if also assumed measurable. On the other hand, in the $\mathbb{F}(\mathbb{C}^n)$-defined analogue the only other possibility is a qualitatively new type of purely-contextual-global symmetry transforming a tuple $(\ell_i)_i$ of lines into $\left(\left(\bigoplus_{j\ne i}\ell_j\right)^{\perp}\right)_i$.
\end{abstract}

\noindent \emph{Key words:
  Grassmannian;
  contextuality;
  eversion;
  frame;
  partition;
  preserver problems;
  projective geometry;
  semilinear maps
}

\vspace{.5cm}

\noindent{MSC 2020: 15A04; 51A05; 51M35; 47B49; 06C15; 51A50; 18A25; 81P13

  % 15A04 Linear transformations, semilinear transformations
  % 51A05 General theory of linear incidence geometry and projective geometries
  % 51M35 Synthetic treatment of fundamental manifolds in projective geometries (Grassmannians, Veronesians and their generalizations)
  % 47B49 Transformers, preservers (linear operators on spaces of linear operators)
  % 06C15 Complemented lattices, orthocomplemented lattices and posets
  % 51A50 Polar geometry, symplectic spaces, orthogonal spaces
  % 18A25 Functor categories, comma categories
  % 81P13 Contextuality in quantum theory
}
%\tableofcontents

%%%%%%%%%%%%%%%%%%%%%%%%%%%%%%%%
%%%%%%%%%%%%%%%%%%%%%%%%%%%%%%%%
\section*{Introduction}

We study and classify maps between either the full Grassmannians $\bG(V)$ of $n$-dimensional Hilbert spaces $V$ or the respective spaces $\bF^{\perp}(V)$ ($\bF(V)$) of ordered orthogonal (respectively independent) line $n$-tuples. The emerging unifying themes, throughout, are the reconstruction of linear or semilinear structure and the rigidity such maps enjoy whenever they preserve context-appropriate partial compatibility. 

The compatibility-to-rigidity dynamic will be familiar to readers acquainted with various \emph{preserver results} exemplified by \cite{MR4927632,2501.06840v2,GogicPetekTomasevic,MR4830482,Petek-TM,zbMATH01100760,MR1311919,MR1866032,Semrl2,zbMATH05302134} and numerous other sources. Such statements share a family resemblance in classifying maps between matrix spaces preserving various algebraic properties or invariants, prominent among these being commutativity and spectra. Commutativity is what here embodies compatibility preservation, and the resulting rigidity emerges as a consequence of spectrum degeneracy (i.e. coincident eigenvalues). Regarding elements of $\bF^{\bullet}$ as geometric counterparts to simple operators (being interpretable as tuples of eigenspaces), spectrum degeneracy can be modeled as \emph{partition linking} among line $n$-tuples: the requirement that the spans of lines indexed by partition blocks agree. This motivates and ties in with the main results to be proved below, summarized here informally for a brief overview:

\begin{itemize}[wide]
\item \clr{} (Commuting-Lattice Rigidity) Dimension-preserving maps $\bG(V)\to \bG(W)$ respecting lattice operations between pairs of spaces whose underlying orthogonal projections commute are induced by semilinear injections. 

\item \pfrp{} (Partition-Frame Rigidity, Perpendicular) Symmetric-group equivariant maps $\bF^{\perp}(V)\to \bF(W)$ respecting partition linking are induced by semilinear injections, linear or conjugate-linear if the maps in question are measurable. 

\item \pfr{} (Partition-Frame Rigidity, General) Symmetric-group equivariant continuous maps $\bF(V)\to \bF(W)$ respecting partition linking are induced by (semi)linear bijections, or such composed with a single additional involution on $\bF(V)$ referred to below as \emph{eversion}. 
\end{itemize}

In preparation for the more formal statements of the three results sketched above, we remind the reader that the \emph{Grassmannian} \cite[\S 5.1]{ms_nonl} of a (here, finite-dimensional) vector space $V$ over a field $\Bbbk$ is 
\begin{equation*}
  \bG(V)
  :=
  \bigsqcup_{0<d\le \dim V}\bG(d,V)
  ,\quad
  \bG(d,V):=\left\{\text{$d$-subspaces of }V\right\}. 
\end{equation*}
For subspaces $V_i\le V$ of a Hilbert space write (cf. also \cite[proof of Theorem 0.2]{2601.01208v1})
\begin{equation*}
  \begin{aligned}
    V_1 \bigobot V_2
    &\iff
      \text{orthogonal projections }P_{V_i}\text{ with ranges $V_i$ commute}\\
    &\iff
      \left(\text{orthogonal complement }V_1\ominus \left(V_1\cap V_2\right)\text{ in }V_1\right)
      \perp
      \left(V_2\ominus \left(V_1\cap V_2\right)\right).
  \end{aligned}
\end{equation*}
If $\bigobot$ holds we also say that $V_i$ \emph{commute} or are \emph{commeasurable} (it is also the \emph{compatibility} relation of \cite[Proposition 1.18]{pank_wign}). The language follows \cite[\S 1]{MR2149209} or \cite[\S 4]{MR4854325}, say, and is intended as reminiscent of the familiar \cite[\S 2-2, pp.77-78]{mack_qm_1963} quantum-mechanical formalism whereby the simultaneously-measurable observables are the commuting self-adjoint Hilbert-space operators.

Recall also \cite[Definition 2.2]{zbMATH01747827} the meaning of \emph{semilinearity} for a map $V\xrightarrow{T}W$ (or \emph{$\alpha$-linearity} when wishing to emphasize $\alpha$) between spaces over fields $\Bbbk_{V}$ and $\Bbbk_{W}$ respectively: additivity as well as 
\begin{equation*}
  \forall\left(v\in V\right)
  \forall\left(c\in \Bbbk_V\right)
  \left(
    Tc v
    =
    \alpha(c)Tv
  \right)
  ,\quad
  \Bbbk_V
  \xrightarrow[\quad\text{field morphism}\quad]{\quad\alpha\quad}
  \Bbbk_W.
\end{equation*}

\begin{theorem}[CLR: Commuting-Lattice Rigidity]\label{th:clr}
  Let $V,W$ be Hilbert spaces of dimension $n\ge 3$ over $\Bbbk\in \left\{\bR,\bC\right\}$ and $\bC$ respectively and $\bG(V)\xrightarrow{\Psi}\bG(W)$ a dimension-preserving map respecting the partial lattice operations for commeasurable pairs.

  $\Psi$ is of the form
  \begin{equation*}
    V
    \ge 
    V'
    \xmapsto{\quad\Psi\quad}
    TV'
    \le W
    ,\quad
    V\lhook\joinrel\xrightarrow[\quad\text{semilinear}\quad]{\quad T\quad}W,
  \end{equation*}
  for $T$ determined uniquely up to scaling. 
\end{theorem}

There is a structural \emph{contextuality} aspect to \clr{}, aligned with how the notion features in mathematical treatments of quantum mechanics \cite[\S 2-2]{mack_qm_1963}. To anchor the discussion, we remind the reader that per \cite[\S B.1]{MR4449330}, say, a \emph{context} attached to a \emph{von Neumann (or $W^*$-)algebra} \cite[\S III.1.8]{blk}
\begin{equation*}
  \cN\le \cL(H)
  :=
  \left(\text{bounded operators on a Hilbert space }H\right)
\end{equation*}
is an abelian $W^*$-subalgebra of $\cN$. 

In \clr{} context refers to families of commeasurable spaces (interpretable quantum-mechanically as simultaneously-measurable yes/no questions), and manifests as the requirement that lattice-operations be preserved only for such pairs. That requirement embodies commutativity via the very definition of commeasurability.

Formalizing partition linking in the earlier overview, we have

\begin{definition}\label{def:pi.rel}
  For a partition $\pi:=\left(\pi_i\right)_{i=1}^s$ of $[n]:=\left\{1..n\right\}=\bigsqcup_{i=1}^s \pi_i$ set
  \begin{equation*}
    \forall\left(\left(\ell_i\right)_i,\left(\ell'_i\right)_i\in \bF^{\bullet}(V)\right)
    \quad:\quad
    \left(\ell_i\right)_i
    \sim_{\pi}
    \left(\ell'_i\right)_i
    \iff
    \forall\left(1\le j\le s\right)
    \left(\bigoplus_{i\in \pi_j}\ell_i=\bigoplus_{i\in \pi_j}\ell'_i\right).
  \end{equation*}
  In words: two line tuples are \emph{$\pi$-linked (or $\pi$-related)} if the subspaces obtained by summing their respective lines over partition index blocks coincide. 
\end{definition}

\begin{theorem}[PFR$_{\perp}$: Partition-Frame Rigidity, Perpendicular]\label{th:pfrp.intro}
  Let $n\in \bZ_{\ge 3}$ and $\Bbbk\in \left\{\bR,\bC\right\}$.
  \begin{enumerate}[(1),wide]
  \item\label{item:th:pfrp.intro:gen} The maps $\bF^{\perp}(\Bbbk^n)\xrightarrow{\Theta}\bF(\bC^n)$ equivariant for the symmetric-group actions permuting lines and preserving $\pi$-relatedness for all partitions $\pi$ of $[n]$ are precisely those of the form 
    \begin{equation}\label{eq:tht.t.intro}
      \bF^{\perp}(\Bbbk^n)
      \ni
      \left(\ell_i\right)_i
      \xmapsto{\quad\Theta=\Theta_T\quad}
      \left(T\ell_i\right)_i
      \in
      \bF(\bC^n)
    \end{equation}
    for semilinear injections $\Bbbk^n\lhook\joinrel\xrightarrow{T}\bC^n$ uniquely determined up to scaling.

  \item\label{item:th:pfrp.intro:cont} The Lebesgue-measurable $\Theta$ in \Cref{item:th:pfrp.intro:gen} are precisely the $\Theta_T$ for linear or conjugate-linear injections $T$. In particular, this holds for continuous $\Theta$. 
  \end{enumerate}
\end{theorem}

\pfrp{} can be regarded as analogous to \emph{Wigner's} classical result \cite[Theorem 4.4]{pank_wign} on maps between Hilbert-space projective spaces, with, again, a contextual layer: the maps $\Theta$ are defined on tuples $\left(\ell_i\right)_1^n$, to be thought of as a line $\ell_1$ (quantum mechanics' familiar \emph{pure states} \cite[\S 2-2, p.75]{mack_qm_1963}), with surrounding context provided by the $\ell_i$, $i\ge 2$. The totality of the lines $\ell_i$, $1\le i\le n$ precisely specify a maximal abelian $*$-subalgebra of $M_n(\Bbbk)$ (those normal operators whose eigenspaces are spans of $\ell_i$), aligning with the aforementioned notion employed in \cite[\S B.1]{MR4449330}.

Enlarging the domain of \Cref{th:pfrp.intro}'s $\Theta$ to the full $\bF$ opens up qualitatively new possibilities. Specifically, the \emph{eversion}\footnote{The term (turning inside-out) is borrowed from its familiar \cite{MR600227} geometric-topology context} self-map of an $\bF^{\bullet}(V)$ space with $\bullet\in \left\{\text{blank},\perp\right\}$ is 
\begin{equation}\label{eq:tht.ev.intro}
  \bF^{\bullet}(V)
  \ni
  \left(\ell_i\right)_i
  \xmapsto{\quad\Theta_{ev}\quad}
  \left(\ell'_i\right)_i
  \in    
  \bF^{\bullet}(V)
  ,\quad
  \ell'_i:=\left(\bigoplus_{j\ne i}\ell_j\right)^{\perp}.
\end{equation}
Naturally, $\Theta_{ev}=\id$ for $\bullet=\perp$, so only blank $\bullet$ will produce an interesting construct. 

\begin{theorem}[PFR: Partition-Frame Rigidity, General]\label{th:pfr.intro}
  For $n\in \bZ_{\ge 3}$ and $\Bbbk\in \{\bR,\bC\}$ the continuous maps $\bF(\Bbbk^n)\xrightarrow{\Theta} \bF(\bC^n)$ equivariant for the symmetric-group actions permuting lines and preserving $\pi$-relatedness for all partitions $\pi$ of $[n]$ are precisely $\Theta=\Theta_T$ or $\Theta_T\circ\Theta_{ev}$ for linear or conjugate-linear bijections $T$.
\end{theorem}

Line tuples once more embody context, provided it is interpreted setup-appropriately: the maximal abelian semisimple algebra generated by those diagonalizable operators whose eigenspaces are spans of the $\ell_i$ for any given $(\ell_i)_i\in \bF(V)$. Although orthogonality is no longer central to the discussion, maximal abelianness still is. 

The novel feature in \pfr{} is eversion: within the ambit of the present discussion, it serves as a purely global-contextual symmetry: meaningless on states (lines alone), requiring the full ambient $(\ell_i)_i$ in order to even make sense of.

The links to preserver theory, moreover, persist surreptitiously: the eversion map $\Theta_{ev}$ and its functorial counterparts in \Cref{def:evrs} below are geometric manifestations of preservers, in which context they retain their somewhat ``exotic'' character. Surprisingly to the authors of \cite[Proposition 2.8]{2501.06840v2}, apart from the usual candidates of (transpose) conjugation, another possibility for commutativity and spectrum-preserving maps defined on various spaces of semisimple operators are those of the form
\begin{equation}\label{eq:exotic.map}
  \Ad_{S} N
  \xmapsto{\quad}
  \Ad_{S^{-1}} N
  ,\quad
  N\text{ normal}
  ,\quad
  S\text{ positive}. 
\end{equation}
These are indeed (perhaps non-obviously) well defined, and may \cite[Theorem 0.4]{2601.01208v1} or may not \cite[Proposition 2.14]{2501.06840v2} be continuous, depending on the parameters of the specific problem. Eversion is precisely what \Cref{eq:exotic.map} boils down to geometrically, upon recasting simple-spectrum operators as their respective eigenline tuples. 

\Cref{se:nat.trnsf.lntpl} below proves \clr{}, reformulated also as \Cref{th:clr.bis}, and recasts \pfrp{} and \pfr{} as \Cref{th:pfrp,th:pfr} respectively. These employ the more involved functorial language of \Cref{not:flg} and \Cref{con:yd} (placing those statements in a suitable categorical perspective).

%%%%%%%%%%%%%%%%%%%%%%%%%%%%%%%%
\subsection*{Acknowledgments}

I am grateful for insightful comments and suggestions from P. \u{S}emrl. 

% % %%%%%%%%%%%%%%%%%%%%%%%%%%%%%%%%
% % %%%%%%%%%%%%%%%%%%%%%%%%%%%%%%%%
% % \section{Preliminaries}\label{se:prel}
% %

%%%%%%%%%%%%%%%%%%%%%%%%%%%%%%%%
%%%%%%%%%%%%%%%%%%%%%%%%%%%%%%%%
\section{Natural transformations between line-tuple functors}\label{se:nat.trnsf.lntpl}

\Cref{th:clr} is a version of \cite[Proposition 2.5]{2501.06840v2}, itself a variant (and reliant on) of the celebrated \emph{Fundamental Theorem of Projective Geometry} \cite[Theorem 3.1]{zbMATH01747827}. 

\pf{th:clr}
\begin{th:clr}
  The current statement generalizes \cite[Proposition 2.5]{2501.06840v2} essentially in removing the continuity assumption, and the proof strategy delivering that result will serve here too; the requisite modifications will affect only the portion of the argument not relying on the usual Fundamental Theorem.

  \begin{enumerate}[(I),wide]
  \item\label{item:th:clr:pf.cncl} \textbf{: Conclusion, assuming $\Psi|_{\bP V:=\bG(1,V)}$ injective.} Said injectivity ensures that the restriction of $\Psi$ to the projective space $\bP V$ of lines in $V$ is a \emph{projective-space morphism} in the sense of \cite[Definition 2.1]{zbMATH01747827}:
    \begin{equation*}
      \forall\left(\ell,\ell',\ell''\in \bP V\right)
      \left(
        \ell''\le \ell+\ell'
        \xRightarrow{\quad}
        \Psi \ell''
        \le
        \Psi \ell+\Psi\ell'
      \right).
    \end{equation*}
    Indeed, the hypotheses imply that $\Psi$ respects inclusions, so that $\Psi \ell+\Psi\ell'$ must be a 2-plane containing $\Psi\ell''$ provided $\Psi\ell$ and $\Psi\ell'$ are distinct.

    Assuming injectivity, then, ensures via \cite[Theorem 3.1]{zbMATH01747827} that the projective-space restriction $\Psi|_{\bP V}$ is induced by a $T$ as in the statement. This then of course extends to $\Psi$ period by the assumed lattice-operation preservation. We are thus left having to prove injectivity on lines.

  \item\label{item:th:clr:red.n3}\textbf{: Reduction to $n=3$.} The problem to be so reduced, recall, is the injectivity of $\Psi|_{\bP V}$. Were $\Psi$ to assign equal values to two distinct lines in $V$, we could simply substitute
    \begin{itemize}[wide]
    \item for $V$, a 3-space $V'\le V$ containing those two lines;
    \item and for $W$, the image $W':=\Psi V'$. 
    \end{itemize}    
    We henceforth assume $n=3$ and seek to prove $\Psi|_{\bP V}$ injective.

  \item\label{item:th:clr:preim.open}\textbf{: Non-singleton preimages of $\Psi|_{\bP V}$ are open.} Assume for a contradiction that $\Psi\ell=\Psi\ell'$ for lines $\ell\ne \ell'\le V$. We observed in \Cref{item:th:clr:pf.cncl} above that
    \begin{equation}\label{eq:plpl}
      \forall\left(\text{lines }\ell_1\ne \ell_2\right)
      \left(
        \Psi \ell_1\ne \Psi\ell_2
        \xRightarrow{\quad}
        \Psi\left(\ell_1+\ell_2\right)
        =
        \Psi\ell_1+\Psi\ell_2
      \right).
    \end{equation}
    Consequently, whenever $\Psi \ell''$ differs from $\Psi\ell$ (and hence also $\Psi\ell'=\Psi\ell$) we have
    \begin{equation}\label{eq:pll}
      \Psi\left(\ell''+\ell\right)
      =
      \Psi\ell''+\Psi\ell
      \xlongequal[\quad]{\quad\text{\Cref{eq:plpl}}\quad}
      \Psi\ell''+\Psi\ell'
      =
      \Psi\left(\ell''+\ell'\right).
    \end{equation}
    This applies in particular to lines $\ell''$ obtained by projecting $\ell'$ orthogonally onto planes $\pi\ni \ell$:
    \begin{equation*}
      \begin{aligned}
        \ell''
        &:=\pi\cap \left(\pi':=\ell'+\pi^{\perp}\right)
          ,\quad
          \text{arbitrary 2-plane $\pi\ge \ell$}\\
        &=\left(\pi:=\ell+\pi'^{\perp}\right)\cap \pi'
          ,\quad
          \text{arbitrary 2-plane $\pi'\ge \ell'$}.
      \end{aligned}
    \end{equation*}
    Or:
    \begin{itemize}
    \item the 2-plane $\pi\ge \ell$ is chosen arbitrarily;
    \item and the 2-plane $\pi'\ge \ell'$ is then uniquely determined by the requirement that $\pi\bigobot \pi'$;
    \item or vice versa (i.e. perform the self-same construction with the roles of $\ell$ and $\ell'$ reversed). 
    \end{itemize}
    Write $(\ell,\ell')/\ell''$ whenever $\ell''$ is obtained from the other two lines by the procedure just described, for some choice of $\pi$. Because $\pi\ne \pi'$ are distinct $\bigobot$-related 2-planes, $\Psi$ cannot conflate them; consequently,
    \begin{equation*}
      (\ell,\ell')/\ell''
      \xRightarrow{\quad\text{\Cref{eq:pll}}\quad}
      \Psi\ell''=\Psi\ell=\Psi\ell'. 
    \end{equation*}
    The set $\Psi^{-1}\Psi\ell\subseteq \bP V$ is thus saturated under the $(-,-)/\bullet$ relation in the sense that whenever that relation's first two arguments belong to said set so does the third. The conclusion now follows from the elementary geometric observation that a $(-,-)/\bullet$-saturated subset of $\bP V$ containing at least two distinct lines is open.
    
  \item\label{item:th:clr:preim.clsd}\textbf{Non-singleton preimages of $\Psi|_{\bP V}$ are closed.} Equivalently, their complements are open. Take $\ell\ne \ell'$ as above, with common $\Psi$-image, and consider $\ell''\not\in \Psi^{-1}\Psi\ell$ which (by openness proven in \Cref{item:th:clr:preim.open}) we may as well assume non-coplanar with $\ell,\ell'$. We then have
    \begin{equation*}
      \Psi\left(\pi:=\ell+\ell''\right)
      =
      \Psi\left(\pi':=\ell'+\ell''\right)
      \ne
      \Psi\pi''
    \end{equation*}
    for \emph{some} 2-plane $\pi\ge \ell''$. Now, \Cref{item:th:clr:preim.open} applied to planes rather than lines proves $\Psi^{-1}\Psi\pi=\Psi^{-1}\Psi\pi'$ open, and
    \begin{equation*}
      \forall\left(\wt{\pi}\in \Psi^{-1}\Psi\pi\right)
      \bigg(
      \Psi\left(\wt{\pi}\cap \pi''\right)
      =
      \Psi\wt{\pi}\cap \Psi\pi''
      =
      \Psi\pi\cap \Psi\pi''
      =
      \Psi\left(\ell''=\pi\cap \pi''\right)
      \bigg).
    \end{equation*}
    The preimage $\Psi^{-1}\Psi\ell''$ is thus a non-singleton, so open by \Cref{item:th:clr:preim.open}. 
    
  \item\label{item:th:clr:cncl}\textbf{Conclusion:} immediate from \Cref{item:th:clr:preim.open} and \Cref{item:th:clr:preim.clsd}, given the connectedness of $\bP V$.  \qedhere
  \end{enumerate}  
\end{th:clr}

We record also the following equivalent version of \Cref{th:clr}.

\begin{theorem}\label{th:clr.bis}
  Let $V,W$ be Hilbert spaces of dimension $n\ge 3$ over $\Bbbk\in \left\{\bR,\bC\right\}$ and $\bC$ respectively and $\bG(V)\xrightarrow{\Psi}\bG(W)$ a map preserving inclusions and the linear independence (but not necessarily the orthogonality) of orthogonal line tuples.

  $\Psi$ is of the form
  \begin{equation*}
    V
    \ge 
    V'
    \xmapsto{\quad\Psi\quad}
    TV'
    \le W
    ,\quad
    V\lhook\joinrel\xrightarrow[\quad\text{semilinear}\quad]{\quad T\quad}W,
  \end{equation*}
  for $T$ determined uniquely up to scaling.  \qedhere
\end{theorem}

\Cref{not:flg} below is a variant of \cite[Notation 1.3]{2601.01208v1}. We adopt familiar terminology (\cite[\S 4.1]{fh_rep-th}, \cite[\S 7.2]{st_en-comb-2_2e_2024}) on \emph{partitions}
\begin{equation*}
  \mu=\left(\mu_1\ge \cdots\ge \mu_s>0\right)
  ,\quad
  \sum_j \mu_j=n
  \quad
  \left(\text{shorthand: }\mu\vdash n\text{ or }|\mu|=n\right)
\end{equation*}
of a positive integer $n$. These have associated \emph{Young} (or \emph{Ferrers}) \emph{diagrams}: $s$ left-justified rows, the $j^{th}$ counting from the top consisting respectively of $\mu_j$ empty boxes. Filling these with the symbols $[n]:=\left\{1..n\right\}$ will produce a \emph{(Young) tableau} $\wh{\mu}\vdash n$ \emph{of shape $\mu$}. This then provides a partition
\begin{equation*}
  [n]=\bigsqcup_j \wh{\mu}_j
  ,\quad
  \wh{\mu}_j:=\left\{\text{symbols on row $j$}\right\}
\end{equation*}
in the usual set-theoretic sense. It is occasionally convenient to pad a partition $\mu=(\mu_j)$ with terminating zeros; should the need arise, we do this tacitly. 

The \emph{reverse refinement order} \cite[post Theorem 13.3]{andr_part_1998} $\wh{\mu} \preceq \wh{\nu}$ for $n:=|\mu|=|\nu|$ is defined by blocks of the $[n]$-partition (associated to) $\wh{\mu}$ being contained in $\wh{\nu}$-blocks.

\begin{remarks}\label{res:why.rev}
  \begin{enumerate}[(1),wide]
  \item By a slight abuse, if convenience suggests that the components of a partition are to be listed in a disordered fashion, they are to be interpreted as having been ordered appropriately; e.g. $(1,3,2)=(3,2,1)\vdash 6$.
    
  \item The reversal in defining `$\preceq$' is a matter of convention, adopted here for its compatibility with the familiar \cite[\S 7.2]{st_en-comb-2_2e_2024} \emph{dominance} or \emph{majorization} order
    \begin{equation*}
      \forall\left(\mu,\nu\vdash n\right)
      \left(
        \mu\le \nu
        \iff
        \forall j
        \left(\mu_1+\cdots+\mu_j\le \nu_1+\cdots+\nu_j\right)
      \right)
    \end{equation*}
    on (numerical rather than set-theoretic) partitions: if $\wh{\mu}\preceq \wh{\nu}$ then $\mu\le \nu$.
  \end{enumerate}  
\end{remarks}

\begin{notation}\label{not:flg}
  \begin{enumerate}[(1),wide]
  \item\label{item:not:flg:fmu}  For an $n$-dimensional $\Bbbk$-vector space $V$ and a partition $\mu\vdash n$ set
    \begin{equation*}
      \bF_{\mu}(V)
      :=
      \left\{
        (V_1,\ \cdots,\ V_s)
        \in
        \prod_{j=1}^s \bG(\mu_j,V)
        \ :\
        \sum_j V_j=V
      \right\}.
    \end{equation*}
    Note that the dimension constraints imply that the components of any $(V_j)\in \bF_{\mu}(V)$ are linearly independent.

  \item\label{item:not:flg:hilb} If furthermore $\Bbbk\in \left\{\bR,\bC\right\}$ and $V$ is equipped with a Hilbert-space structure, $\bF_{\mu}^{\perp}(V)$ denotes the analogous space consisting of mutually-orthogonal tuples $(V_j)$ (so that $\bF^{\perp}(V)\subseteq \bF(V)$). 
    
  \item\label{item:not:flg:rfnmt} Every reverse refinement $\wh{\mu}\preceq \wh{\nu}$ induces a map
    \begin{equation}\label{eq:fmunu}
      \bF^{\bullet}_{\mu}(V)
      \ni
      (V_j)_{j=1}
      \xmapsto{\quad \bF^{\bullet}_{\wh{\mu}\preceq \wh{\nu}} = \bF^{\bullet}_{\wh{\mu}\preceq \wh{\nu}}(V)\quad}
      \left(
        W_k
        :=
        \sum_{\wh{\mu}_j\subseteq \wh{\nu}_k} V_j
      \right)_k
      \in
      \bF^{\bullet}_{\nu}(V);
    \end{equation}
    in words, the refinement indicates which components of $(V_j)$ to sum in order to form the components of its image $(W_k)_k$.
  \end{enumerate}
\end{notation}

\begin{remarks}\label{res:sym.gp.act}
  \begin{enumerate}[(1),wide]
  \item\label{item:res:sym.gp.act:rc} We work mostly over the real or complex numbers, in which case $\bF^{\bullet}_{\mu}(V)$ are also equipped with their usual (real/complex-manifold) topologies inherited from those of the Grassmannians. All references to continuity or weaker regularity properties such as being \emph{Borel} or (Lebesgue) \emph{measurable} \cite[\S\S 1.11, 2.20]{rud_rc_3e_1987}, in the context of discussing $\bF^{\bullet}_{\mu}$, will assume that setup.

  \item\label{item:res:sym.gp.act:sn.acts} For a partition $\mu$ denote by
    \begin{equation*}
      \mathrm{jmp}(\mu,1)
      ,\quad
      \mathrm{jmp}(\mu,2)
      ,\quad\cdots
    \end{equation*}
    the strictly-positive members of the sequence $(\mu_j-\mu_{j+1})_j$, listed $j$-increasingly. Each $F^{\bullet}_{\mu}$ comes equipped with a free continuous action by the product
    \begin{equation*}
      S_{\mu}
      :=
      \prod_i S_{\mathrm{jmp}\left(\mu^{\perp},i\right)}
      ,\quad
      S_m
      :=
      \text{symmetric group on $m$ symbols},
    \end{equation*}
    where $\mu^{\perp}$ is the partition \emph{conjugate} \cite[Notation]{fult_y_1997} to $\mu$: in Young-diagram terms, one's rows are the other's columns.

    The first factor $S_{\mathrm{jmp}\left(\mu^{\perp},1\right)}$ permutes the largest terminal segment of equidimensional $V_j$ listed in $(V_j)_j\in \bF^{\bullet}_{\mu}$, the next-to-last $S_{\bullet}$ permutes the next maximal segment of equidimensional components, etc.

  \item In the context of \Cref{not:flg}\Cref{item:not:flg:rfnmt}, the fixed reverse refinement $\wh{\mu}\preceq \wh{\nu}$  also provides an embedding $S_{\nu}\le S_{\mu}$, making both the domain and codomain of \Cref{eq:fmunu} into $S_{\nu}$-spaces; that map is easily checked to be equivariant for the two respective $S_{\nu}$-actions. 
  \end{enumerate}
\end{remarks}

\begin{construction}\label{con:yd}
  Having fixed $V$ (so that $\bF$ means $\bF(V)$ throughout), the varying subscripts of $\bF_{\cdot}^{\bullet}$ make the latter into functors
  \begin{equation*}
    \cY\cD
    =
    \cY\cD_{n:=\dim V}
    \xrightarrow{\quad \bF_{\cdot}^{\bullet}\quad}
    \left[
      \begin{gathered}
        \cat{Set}\left(\text{in general}\right)\\
        \left.
          \begin{gathered}
            \cat{Top}\left(\text{topological spaces}\right)\\
            \cat{Meas}\left(\text{\emph{measurable spaces}: \cite[Definition 1.3]{rud_rc_3e_1987}}\right)        
          \end{gathered}
        \right\}
        \text{ if $\Bbbk\in \left\{\bR,\bC\right\}$}
      \end{gathered}
    \right.
  \end{equation*}
  for an appropriately-defined category $\cC$: the objects are young diagrams $\mu\vdash n$, the morphisms $\mu\to \nu$ are the inverse refinements $\wh{\mu}\preceq \wh{\nu}$, and the composition is the unique, obvious one.

  It thus makes sense, in this context, to refer to the usual functor-theoretic constructs attached to $\bF^{\bullet}_{\cdot}(V)$ (e.g. natural transformations/isomorphisms between two such, for distinct $V$, $V'$). 
\end{construction}

\begin{remark}\label{re:2cat.constr.yd}
  \Cref{con:yd}'s category $\cY\cD=\cY\cD_n$ can be recovered via familiar 2-categorical constructs:
  \begin{itemize}[wide]
  \item Consider the dominance-ordered poset $\cP$ of partitions $\vdash n$ as a category in the usual \cite[Example 1.2.6.b]{brcx_hndbk-1} manner, with one arrow $\to$ corresponding to every $\le$.

  \item The same goes for the poset $\left(\cT,\preceq\right)$ of tableaux $\vdash n$ under reverse refinement.

  \item There is a functor (i.e. poset morphism) $\cT\xrightarrow{U}\cP$ forgetting tableau fillings.

  \item This produces the \emph{comma category} \cite[Definition 1.6.1]{brcx_hndbk-1} $U\downarrow \id_{\cP}$, whose objects are pairs
    \begin{equation*}
      \left(\wh{\mu}\in \cT,\ \nu\in \cP\right)
      \quad\text{with}\quad
      \mu\le \nu.
    \end{equation*}

  \item There is a natural transformation
    \begin{equation*}
      \begin{tikzpicture}[>=stealth,auto,baseline=(current  bounding  box.center)]
        \path[anchor=base] 
        (0,0) node (l) {$U\downarrow\id_{\cP}$}
        +(4,0) node (r) {$\cP$}
        ;
        \draw[->] (l) to[bend left=20] node[pos=.4,auto] {$\scriptstyle \left(\wh{\mu},\nu\right)\mapsto \mu$} (r);
        \draw[->] (l) to[bend right=20] node[pos=.4,auto,swap] {$\scriptstyle \left(\wh{\mu},\nu\right)\mapsto \nu$} (r);
        \draw[-implies,double equal sign distance] (2,.3) to node[pos=.4,auto] {$\scriptstyle \alpha$} (2,-.1);
      \end{tikzpicture}
    \end{equation*}
    induced by $\mu\le \nu$;

  \item Finally, set $\cY\cD:=\cat{Inv}(\alpha)$: the \emph{inverter} of $\alpha$, i.e. \cite[Exercise 2.m]{ar} the full subcategory of $U\downarrow\id_{\cP}$ over which $\alpha$ is an isomorphism.
  \end{itemize}
\end{remark}

\begin{theorem}\label{th:pfrp}
  Let $n\in \bZ_{\ge 3}$ and $\Bbbk\in \{\bR,\bC\}$.
  \begin{enumerate}[(1),wide]
  \item\label{item:th:pfrp:gen} The natural transformations $\bF_{\cdot}^{\perp}(\Bbbk^n)\xrightarrow{\Theta} \bF_{\cdot}(\bC^n)$ of $\cat{Set}$-valued functors, symmetric in the sense of equivariance with respect to the $S_{\bullet}$-actions in \Cref{res:sym.gp.act}\Cref{item:res:sym.gp.act:sn.acts}, are precisely those of the form
    \begin{equation}\label{eq:tht.t}
      \bF_{\mu}^{\perp}(\Bbbk^n)
      \ni
      \left(V_i\right)_i
      \xmapsto{\quad\Theta=\Theta_T\quad}
      \left(TV_i\right)_i
      \in
      \bF_{\mu}(\bC^n)
    \end{equation}
    for semilinear injections $\Bbbk^n\lhook\joinrel\xrightarrow{T}\bC^n$.

  \item\label{item:th:pfrp:cont} The $\cat{Top}$- or $\cat{Meas}$-valued natural transformations $\bF_{\cdot}^{\perp}(\Bbbk^n)\xrightarrow{\Theta} \bF_{\cdot}(\bC^n)$ are $\Theta=\Theta_T$ for linear or conjugate-linear injections $\Bbbk^n\lhook\joinrel\xrightarrow{T}\bC^n$.
  \end{enumerate}
\end{theorem}

% % OLD: NOW INLINE IN THE PROOF OF \Cref{th:pfrp} 
% % 
% % We will be referring to binary relations on the spaces $\bF^{\bullet}_{\mu}(V)$ extending (as will become apparent) \Cref{th:clr}'s $\bigobot$.
% % 
% % \begin{definition}\label{def:comm}
% %   Set
% %   \begin{equation*}
% %     \forall\left(\left(V_{\pm i}\right)_i\in \bF^{\bullet}_{\mu}(V)\right)
% %     \quad:\quad
% %     \left(V_{+i}\right)
% %     \bigobot
% %     \left(V_{-i}\right)
% %     \xLeftrightarrow{\ \text{def}\ }
% %     \forall j
% %     \left(V_{+j}=\bigoplus_i\left(V_{-i}\cap V_{+j}\right)\right).
% %   \end{equation*}
% %   The relation is symmetric, and the notation is meant as suggestive of the fact that $W\bigobot W'$ in the sense of \Cref{th:clr} if and only if $(W,W^{\perp})\bigobot (W',W'^{\perp})$ in the present sense.
% %   
% %   If $\bigobot$ holds we also say that $\left(V_{\pm i}\right)_i$ \emph{commute} or are \emph{commensurable}. The language is reminiscent of the familiar \cite[\S 2-2, pp.77-78]{mack_qm_1963} quantum-mechanical formalism whereby the simultaneously-measurable\footnote{Hence also the variant \emph{commeasurable} \cite[\S 4]{MR4854325}.} observables are the commuting self-adjoint Hilbert-space operators.
% % \end{definition}

\begin{proof}
  That all $\Theta_T$ do satisfy the requirements is not in question, so we only address the converse. 
  
  \Cref{item:th:pfrp:cont} follows from \Cref{item:th:pfrp:gen}: the measurability of $\Theta$ (or rather, of its individual component maps) entails that of $T$ and the underlying field embedding $\Bbbk\lhook\joinrel\xrightarrow{\alpha}\bC$ rendering $T$ semilinear. One concludes by noting that the only continuous field embeddings are the usual $\bR\lhook\joinrel\to \bC$ if $\Bbbk=\bR$ and the identity and complex conjugation otherwise, for instance as a consequence of the automatic linearity \cite[Theorem 9.4.3]{kczm_func-eq_2e_2009} of all measurable additive maps $\bR\to \bR$. 
  
  As to \Cref{item:th:pfrp:gen}, observe that the functoriality of $\Theta$ and its assumed symmetry jointly ensure that it specializes to
  \begin{equation*}
    \bF_{(d,n-d)}^{\perp}(\Bbbk^n)
    \ni
    \left(U,U^{\perp}\right)
    \xmapsto{\quad\Theta\quad}
    \left(\Psi U,\Psi \left(U^{\perp}\right)\right)
    \in
    \bF_{(d,n-d)}(\bC^n)
  \end{equation*}
  for a map $\bG(\Bbbk^n)\xrightarrow{\Psi}\bG(\bC^n)$. Note next that (again, by symmetry/functoriality) $\Theta$ will preserve the binary relation
  \begin{equation}\label{eq:obot.gen}
    \forall\left(
      \left(V_{\pm i}\right)_i\in \bF^{\bullet}_{\pm\mu}(V)
      \text{ respectively}
    \right)
    \quad:\quad
    \left(V_{+i}\right)
    \bigobot
    \left(V_{-i}\right)
    \xLeftrightarrow{\ \text{def}\ }
    \forall j
    \left(V_{+j}=\bigoplus_i\left(V_{-i}\cap V_{+j}\right)\right)
  \end{equation}
  extending (as the notation suggests: \Cref{re:obot.gen}) \Cref{th:clr}'s $\bigobot$. This in turn implies that $\Psi$ meets the hypotheses of \Cref{th:clr}, and that earlier result delivers the conclusion.   
\end{proof}

\begin{remark}\label{re:obot.gen}
  The relation $\bigobot$ of \Cref{eq:obot.gen} is symmetric, and the notation is meant as suggestive of the fact that $W\bigobot W'$ in the sense of \Cref{th:clr} if and only if $(W,W^{\perp})\bigobot (W',W'^{\perp})$ in the present sense.  
\end{remark}

% % OLD: MOVED UP TO INTRO 
% % 
% % Enlarging the domain of \Cref{th:pfrp}'s $\Theta$ to the full $\bF_{\cdot}$ opens up qualitatively new possibilities. Specifically:
% %

The functorial version of \Cref{eq:tht.ev.intro} is as expected. 

\begin{definition}\label{def:evrs}
  The \emph{eversion} %\footnote{The term (turning inside-out) is borrowed from its familiar \cite{MR600227} geometric-topology context}
  natural (self-)transformation of an $\bF^{\bullet}_{\cdot}(V)$ functor is 
  \begin{equation}\label{eq:tht.ev}
    \bF_{\mu}^{\bullet}(V)
    \ni
    \left(V_i\right)_i
    \xmapsto{\quad\Theta_{ev}\quad}
    \left(V'_i\right)_i
    \in    
    \bF_{\mu}^{\bullet}(V)
    ,\quad
    V'_i:=\left(\bigoplus_{j\ne i}V_j\right)^{\perp}.
  \end{equation}
  As in \Cref{eq:tht.ev.intro}, $\Theta_{ev}=\id$ for $\bullet=\perp$.
\end{definition}

\begin{theorem}\label{th:pfr}
  For $n\in \bZ_{\ge 3}$ and $\Bbbk\in \{\bR,\bC\}$ the $\cat{Top}$-valued natural transformations $\bF_{\cdot}(\Bbbk^n)\xrightarrow{\Theta} \bF_{\cdot}(\bC^n)$ are precisely those falling into one of the following classes.
  \begin{enumerate}[(a),wide]
  \item $\Theta=\Theta_T$ as in \Cref{eq:tht.t} for linear or conjugate-linear injections $\Bbbk^n\lhook\joinrel\xrightarrow{T}\bC^n$;

  \item or of the form $\Theta_T\circ\Theta_{ev}$, with the latter factor as defined in \Cref{eq:tht.ev}. 
  \end{enumerate}
\end{theorem}

\begin{remark}\label{re:evers.order}
  The reader might naturally wonder at this stage whether one should not allow for more sophisticated compositions of $\Theta_T$ and $\Theta_{ev}$ in the statement of \Cref{th:pfr}: plainly, the natural transformations with the requisite properties compose well. The simple remark recorded as \Cref{le:evers.order} below elucidates the matter. 
\end{remark}

\begin{lemma}\label{le:evers.order}
  For $\Bbbk\in \left\{\bR,\bC\right\}$ a composition $\Theta_{ev}\circ\Theta_T$ for invertible (conjugate-)linear $T\circlearrowright \Bbbk^n$ is of the form $\Theta_{T}\circ \Theta_{ev}$ for an invertible (respectively conjugate-)linear $T'$. 
\end{lemma}
\begin{proof}
  We focus on the linear complex case to fix ideas; little would change in addressing the other branches. The observation follows from
  \begin{equation*}
    \begin{aligned}
      \Theta_{ev}\circ\Theta_T
      &\xlongequal[\quad\substack{\text{$U$ unitary}\\\text{$P$ positive}\\\text{\emph{polar decomposition} \cite[\S I.5.2.2]{blk}}}\quad]{\quad T=UP\quad}
      \Theta_U\circ \Theta_{ev}\circ\Theta_P\\
      &=
        \Theta_U\circ \Theta_{P^{-1}}\circ\Theta_{ev}
        =
        \Theta_{UP^{-1}}\circ \Theta_{ev}
    \end{aligned}
  \end{equation*}
  upon setting $T':= UP^{-1}$. 
\end{proof}

\pf{th:pfr}
\begin{th:pfr}
  We know from \Cref{th:pfrp}\Cref{item:th:pfrp:cont} that for every $S\in \GL(\Bbbk^n)$ the restriction $\Theta|_{S\bF^{\perp}_{\cdot}}$ is of the form $\Theta_{T_S}$ for some linear or conjugate-linear invertible $T_S\circlearrowright \bC^n$ (defined uniquely only up to scaling, which ambiguity will however make no material difference). Composing with $\Theta^{-1}_{T_{\id}}$, we can (and will, throughout the rest of the proof) assume that $\Theta|_{\bF^{\perp}_{\cdot}}$ is induced by the standard embedding $\Bbbk^n\le \bC^n$. Henceforth conflating, by a slight abuse, that standard embedding with an identity, the goal is to show that $\Theta\in \left\{\id,\Theta_{ev} \right\}$. Equivalently:
  \begin{equation}\label{eq:all.or.all}
    \forall \left(S\in \GL(\Bbbk^n)\right)
    \left(\Theta|_{S\bF^{\perp}_{\cdot}}=\id\right)
    \quad
    \vee
    \quad
    \forall\left(S\in \GL(\Bbbk^n)\right)
    \left(\Theta|_{S\bF^{\perp}_{\cdot}}=\Theta_{ev}\right)
  \end{equation}

  A number of moves will gradually simplify the problem.

  \begin{enumerate}[(I),wide]

    % % OLD: STEP NO LONGER NEEDED.
    % % 
    % % \item\textbf{: All $T_S$ are linear.} Note that $T_S$ is well-defined only modulo scalars (the pointwise stabilizer of $S\bF_{\cdot}^{\perp}\subseteq \bF_{\cdot}$), but this will make no material difference to the discussion.
    % %   
    % %   The two sets
    % %   \begin{equation}\label{eq:sts}
    % %     \left\{S\ :\ T_S\text{ (conjugate-)linear}\right\}
    % %   \end{equation}
    % %   are both closed, and it suffices for $S$ to range over the identity connected component
    % %   \begin{equation*}
    % %     \GL_0(\Bbbk^n)
    % %     \le
    % %     \GL(\Bbbk^n)
    % %   \end{equation*}
    % %   (for $\GL_0 \bF_{\cdot}^{\perp}=\GL \bF_{\cdot}^{\perp}$). Connectedness and the closure of both \Cref{eq:sts} ensure that exactly one of those sets can be non-empty; it must be the linear one, for $T_{\id}=\id$. 
    % %   

  \item\textbf{: Reduction to positive $S$.} That is, if and when convenient, we can assume $S$ to be a \emph{positive} operator in the usual sense \cite[\S I.4.2]{blk} of being of the form $S=X^*X$. Indeed:
    \begin{equation*}
      \GL_+(\Bbbk^n)
      \bF_{\cdot}^{\perp}
      =\GL(\Bbbk^n) \bF_{\cdot}^{\perp},
      \quad
      \GL_+(\Bbbk^n)
      :=
      \left\{S\in \GL(\Bbbk^n)\ :\ S\ge 0\right\}.
    \end{equation*}
    
  \item\textbf{: Quantifier/connective flip.} That is, we can substitute for \Cref{eq:all.or.all} the formally weaker
    \begin{equation*}
      \forall\left(S\in \GL_+\right)      
      \left(
        \Theta|_{S\bF^{\perp}_{\cdot}}=\id
        \ 
        \vee
        \ 
        \Theta|_{S\bF^{\perp}_{\cdot}}=\Theta_{ev}
      \right).
    \end{equation*}
    This is again a connectedness argument: the two sets
    \begin{equation}\label{eq:all.or}
      \left\{
        S\in \GL_+(\Bbbk^n)
        \ :\
        \Theta|_{S\bF^{\perp}_{\cdot}}
        =
        \left[
          \begin{aligned}
            &\id\\
            &\Theta_{ev}
          \end{aligned}
        \right.
      \right\},
    \end{equation}
    partitioning the connected $\GL_+$ if \Cref{eq:all.or} holds, are both closed; one is thus empty as soon as the other is not.

  \item\textbf{: It suffices to prove the union of the two sets \Cref{eq:all.or} open.} For indeed, that union would then be clopen and non-empty (as $S:=\id$ belongs to it), hence all of $\GL_+$ by connectedness.
    
  \item\textbf{: Taking stock.} The problem has now been reduced to 
    \begin{equation*}
      \forall\left(S\in U\right)      
      \left(
        \Theta|_{S\bF^{\perp}_{\cdot}}=\id
        \ 
        \vee
        \ 
        \Theta|_{S\bF^{\perp}_{\cdot}}=\Theta_{ev}
      \right).
    \end{equation*}
    for some neighborhood $U=\overset{\circ}{U}\ni S_0$ in $\GL_+$ of any $S_0\in\bigcup\text{\Cref{eq:all.or}}$. Nothing substantive is lost by taking $S_0=\id$ to fix ideas, so the goal henceforth will be to argue that for some sufficiently small $\id$-neighborhood $U=\overset{\circ}{U}\subseteq \GL_+$ we have
    \begin{equation*}
      \forall \left(S\in U\right)
      \left(\Theta|_{S\bF^{\perp}_{\cdot}}=\id\right)
      \quad
      \vee
      \quad
      \forall\left(S\in U\right)
      \left(\Theta|_{S\bF^{\perp}_{\cdot}}=\Theta_{ev}\right).
    \end{equation*}        

  \item\label{item:th:pfr:pf.bicom}\textbf{: $S$ and $\wh{S}:=T_S S$ bicommute.} The phrase borrows from operator-algebraic parlance \cite[\S I.2.5.3]{blk}, and here means $\wh{S}$ commutes with all (orthogonal) projections commuting with $S$. To that end, observe that 
    \begin{equation*}
      \forall \left((V_i)\in \bF_{\cdot}^{\perp}(\Bbbk^n)\right)
      \bigg(
      \left(SV_i\right) = (V_i)
      \xRightarrow{\quad}
      \left(\wh{S}V_i\right) = \left(SV_i\right) = (V_i)
      \bigg);
    \end{equation*}
    this precisely translates the target claim.
    
  \item\textbf{: Proof proper.} The binary relation $\bigobot$ of \Cref{eq:obot.gen} is again preserved by $\Theta$, so in particular
    \begin{equation*}%\label{eq:vswvtsw}
      \forall \left((V_i),\ (W_i)\in \bF_{\cdot}^{\perp}(\Bbbk^n)\right)
      \bigg(
      (V_i)\bigobot \left(SW_i\right)
      \quad
      \xRightarrow{\quad}
      \quad
      (V_i)\bigobot \left(\wh{S} W_i\right)
      \bigg).
    \end{equation*}    
    Apply this remark to the following setup:
    \begin{itemize}[wide]
    \item $(W_i)=\left(\ell^{\perp},\ell\right)\in \bF_{(n-1,1)}^{\perp}$ for varying lines $\ell\le \Bbbk^n$;

    \item and
      \begin{equation*}
        (V_i)
        :=
        \left(V^{\perp},\ V:=\spn\left\{S\ell,\ \left(S\ell^{\perp}\right)^{\perp}=S^{-1}\ell\right\}\right)
        \in
        \bF_{(n-2,2)}^{\perp}
        \text{ or }
        \bF_{(n-1,1)}^{\perp}.
      \end{equation*}
    \end{itemize}
    The latter option $\bF_{(n-1,1)}^{\perp}$ occurs only exceptionally, when $S\ell$ and $S^{-1}\ell$ happen to coincide; generically (which is the case we can focus on by density) $(V_i)\in \bF_{(n-2,2)}^{\perp}$. 
    
    For $S$ sufficiently close to $\id$ the operator $\wh{S}\sim \id\sim  S$ itself is linear (rather than conjugate-linear) and the line $\wh{S} \ell$, itself being close to $S\ell$ and hence $\ell$, cannot be contained in $V^{\perp}$. Consequently, for $S$ ranging over an open dense subset of an $\id$-neighborhood in $U=\overset{\circ}{U}\subseteq \GL_+(\Bbbk^n)$,
    \begin{equation*}
      \forall \ell
      \left(\wh{S} \ell\in \spn\left\{S{\ell},S^{-1}\ell\right\}\right)
      \xRightarrow[\quad]{\quad\text{\Cref{item:th:pfr:pf.bicom} and \cite[Theorem 2.6]{MR1897909}}\quad}
      \wh{S}\in \spn\left\{S^{\pm 1}\right\}.
    \end{equation*}
    The hypothesis holds also for the line
    \begin{equation*}
      \left(\wh{S} \ell^{\perp}\right)^{\perp}
      =
      \wh{S}^{*-1}\ell
    \end{equation*}
    (the last exponent denoting the inverse adjoint), so that
    \begin{equation*}
      \forall\left(S\in U\right)
      \left(\wh{S},\wh{S}^{*-1}\in \spn\left\{S^{\pm 1}\right\}\right).
  \end{equation*}
  This is sufficient to conclude that $\wh{S}\in \left\{S^{\pm 1}\right\}$ up to scaling: this is effectively what the argument yielding \cite[Proposition 2.10]{2501.06840v2} proves. One of the options $\wh{S}\in \bC S^{\pm 1}$ obtains throughout a dense subset of $U$ and hence over $U$ generally, yielding the conclusion.  \qedhere
  \end{enumerate}
\end{th:pfr}

%%%%%%%%%%%%%%%%%%%%%%%%%%%%%%%% 
%%%%%%%%%%%%%%%%%%%%%%%%%%%%%%%%

\addcontentsline{toc}{section}{References}
%\bibliography{bib}{}
%\bibliographystyle{plain}

\def\polhk#1{\setbox0=\hbox{#1}{\ooalign{\hidewidth
  \lower1.5ex\hbox{`}\hidewidth\crcr\unhbox0}}}

\Addresses

\end{document}